\documentclass[10pt]{article}
\usepackage{amssymb,cite}
\usepackage{amsmath}
\usepackage{epsfig}
\usepackage{graphicx}
\usepackage[all]{xy}
\usepackage{color}
\usepackage{enumitem} 
\usepackage{kotex}
\usepackage[usenames,dvipsnames]{xcolor}
\usepackage[color=green!40]{todonotes}
\usepackage{setspace}
\usepackage{authblk,amsthm,amsfonts,amssymb,epsfig,graphics,amsmath,amsbsy,subfigure}

\begin{document}
	\def\mathbi#1{\textbf{\em #1}}
	\newcommand{\intL}{\int\limits}
	\newcommand{\half}{^\infty_0 }
	\newcommand{\intR}{\int\limits_{\mathbb{R}} }
	\newcommand{\intRR}{\int\limits_{\mathbb{R}^2} }
	\newcommand{\RR}{\mathbb{R}}
	\newcommand{\xx}{\mathbf{x}}
	\newcommand{\XX}{\mathbf{X}}
	\newcommand{\yy}{\mathbf{y}}
	\newcommand{\YY}{\mathbf{Y}}
	\newcommand{\zz}{\mathbf{z}}
	\newcommand{\uu}{\mathbf{u}}
	\newcommand{\xb}{\textbf{\textit{x}}} 
	\newcommand{\yb}{\textbf{\textit{y}}} 
	\newcommand{\zb}{\textbf{\textit{z}}}
	\newcommand{\ub}{\textbf{\textit{u}}}
	\newcommand{\ttheta}{\boldsymbol{\theta}}
	\newcommand{\pphi}{\boldsymbol{\phi}}
	\newcommand{\oomega}{\boldsymbol{\omega}}
	\newcommand{\eeta}{\boldsymbol{\eta}}
	\newcommand{\zzeta}{\boldsymbol{\zeta}}
	\newcommand{\cchi}{\boldsymbol{\chi}}
	\newcommand{\xxi}{{\boldsymbol{\xi}}}
	\newcommand{\rmd}[1]{\mathrm d#1}
	\newtheorem{thm}{Theorem}
	\newtheorem{defi}[thm]{Definition}
	\newtheorem{rmk}[thm]{Remark}
	\newtheorem{smy}[thm]{Summary}
	\newtheorem{cor}[thm]{Corollary}
	\newtheorem{lem}[thm]{Lemma}
	\newtheorem{prop}[thm]{Proposition}
	
	
	\title{Reconstruction of the initial function from the solution of the fractional wave equation measured in two geometric settings}
	\author[a]{Hyungyeong Jung}
	\author[b,*]{Sunghwan Moon}
	\affil[a]{\small{School of Mathematics, Kyungpook National University, Daegu 41566, Republic of Korea} }
	\affil[b]{\small{Department of Mathematics, College of Natural Sciences, Kyungpook National University, Daegu 41566, Republic of Korea\\
	sunghwan.moon@knu.ac.kr} }
\date{}
	\maketitle
	\begin{abstract}
		Photoacoustic tomography (PAT) is a novel and rapidly promising technique in the field of medical imaging, based on the generation of acoustic waves inside an object of interest by stimulating non-ionizing laser pulses. This acoustic wave is measured using the detector on the outside of the object and converted into an image of the human body by several inversions. Thus, one of mathematical problems in PAT is how to recover the initial function from the solution of the wave equation on the outside of the object. In this study we consider the fractional wave equation and assume that the point-like detectors are located on the sphere and hyperplane. We provide how to recover the initial function from the data, the solution of the fractional wave equation, measured on the sphere and hyperplane.\\

\noindent 	Keywords: \texttt{Photoacoustic, tomography, wave equation, fractional derivative} 
	\end{abstract}

	
	\section{Introduction}
	
	
    PhotoAcoustic Imaging (PAI) is a new biomedical imaging modality that secures the advantages of each while complementing problems of optical and ultrasound imaging. It is a hybrid technology that combines the high-contrast and spectroscopic-based specificity of optical imaging with the high spatial resolution of ultrasound imaging \cite{beard11,jiang11}. PAI uses on photoacoustic effects to form images of biological tissues without tissue damage. The photoacoustic effect, discovered by Alexander Graham Bell in 1880, refers to the generation of acoustic waves using thermal expansion by absorbing electromagnetic waves such as light or radio wave \cite{bayer12,bell80}.
    
	
	Photoacoustic tomography (PAT) is a PAI system such that, a non-ionizing pulse wave with a strong intensity and very short irradiation time is irradiated to the tissue for diagnosis to obtain a photoacoustic signal in the ultrasound range (several MHz to several tens of MHz). The photoacoustic signal is an acoustic signal generated during thermal expansion, produced by irradiating a laser to the tissue and absorbing the irradiated laser energy by the tissue. Therefore, the generated Photoacoustic signal is received using an ultrasonic detector and an image is formed through some inversions with the received signal.
	
	
 	
 	One of the mathematical problems arising in PAT is how to recover the initial function from the data measured on the outside of the object. 
 	Here the measurement data satisfy the wave equation because the signal is an ultrasonic wave. According to \cite[chapter 3.]{hilfer10}, solutions of fractional order differential equations better describe real-life situations compared with those of the corresponding integer-order differential equations.
	In this paper, we consider the initial value problem for the fractional wave equation:
	\begin{equation}\label{eq:pdeofpat}
	\begin{array}{ll}
	D^{\alpha}_{t}p_{\alpha}(\xx,t)=-(-\Delta_{\xx})^{\frac{\alpha}{2}}p_{\alpha}(\xx,t) & (\xx,t)\in\RR^{n}\times[0,\infty),\ 1<\alpha\le2 \\
	p_{\alpha}(\xx,0)=f(\xx) & \xx\in\RR^n \\
	\partial_{t}p_{\alpha}(\xx,t)|_{t=0}=0 & \xx\in\RR^{n}
	\end{array}
	\end{equation}
	where $-(-\Delta_{\xx})^{\frac{\alpha}{2}}$ is the Riesz space-fractional derivative of order $\alpha$  defined below, and $D^{\alpha}_{t}$ is the Caputo time-fractional derivative of order $\alpha$,	
	\begin{equation*}
	(D^{\alpha}_{t} h)(t):=(I^{m-\alpha}h^{(m)})(t), \quad m-1<\alpha\le m,\ m\in\mathbb N,
	\end{equation*}
	$I^\alpha,$ $\alpha\ge 0$ is the Riemann-Liouville fractional integral
	\begin{equation*}
	(I^{\alpha} h)(t):=\left\{
	\begin{array}{ll}
	\dfrac{1}{\Gamma(\alpha)} \displaystyle\intL^{t}_0(t-\tau)^{\alpha-1}h(\tau){\rm d}\tau,\quad & \mbox{if} \ \alpha>0, \\
	h(t),\quad & \mbox{if} \ \alpha=0, 
	\end{array}
	\right. 
	\end{equation*}
	and $\Gamma(\cdot)$ is the gamma function. For $\alpha=m$, $m\in\mathbb N$, the Caputo fractional derivative coincides with the standard derivative of order $m$. 
	For a smooth function $f$ on $\RR^n$ with compact support, the Riesz fractional derivative\cite{saichev97, samko93} of order $\alpha$, $0\le\alpha$ is defined as 
	\begin{equation*}
	\mathcal F(-(-\Delta_{\xx})^{\frac\alpha2}f)(\xxi) := -|\xxi|^{\alpha}(\mathcal F f)(\xxi),
	\end{equation*}
	$\mathcal F$ is the Fourier transform of a function $f$ defined by 
	\begin{equation*}
	(\mathcal F f)(\xxi):=\intL_{\RR^{n}}f(\xx)e^{-\mathrm{i}\xx\cdot\xxi}{\rm d}\xx. 
	\end{equation*}
	The solution of the fractional wave equation (1) is  
	\begin{equation*}\label{eq:solpde}
	p_{\alpha}(\xx,t)=\frac{1}{(2\pi)^n}\intL_{\RR^{n}} E_\alpha(-t^\alpha |\xxi|^{\alpha})e^{\mathrm{i}\xxi\cdot\xx}\mathcal F f(\xxi){\rm d}\xxi.
	\end{equation*}
	Here 
	\begin{equation*}
	E_\alpha(z)=\sum^\infty_{k=0}\frac{z^k}{\Gamma(1+\alpha k)}, \quad\alpha>0, \ z \in \mathbb{C},
	\end{equation*}
	is the Mittag-Leffler function. From \cite[Lemma 2.23]{kilbasst06}
	\begin{equation*}
	D^\alpha_tE_\alpha(-t^\alpha |\xxi|^{\alpha} )=-|\xxi|^{\alpha}E_\alpha(-t^\alpha |\xxi|^{\alpha} )
	\end{equation*}
	and
	\begin{equation*}
	D_tE_\alpha(-t^\alpha|\xxi|^{\alpha} )|_{t=0}=0. 
	\end{equation*}
		Since $E_2(-z^2)=\cos(z)$ for $\alpha=2$, the solution $p_{\alpha}$ of \eqref{eq:pdeofpat} reduces to the solution of the wave equation.
    Therefore, we focus on the case $1<\alpha<2$ since for the case $\alpha=2$ is well studied in many literatures \cite{bukhgeimk78,finchhr07,kunyansky12,kuchment14book,moon18,narayananr10,xuw05,zangerl19}. 
	\section{Preliminary} 
	
	Here, we consider two geometries where point-like detectors are located namely, spherical and hyperplanar geometries. As their names imply, in each case,  detectors are located on the unit sphere and hyperplane, respectively (see Figure \ref{fig:geom}). Our goal is to reconstruct the initial function $f$ from the measurement data, i.e., the solution of \eqref{eq:pdeofpat} on two geometries. 
    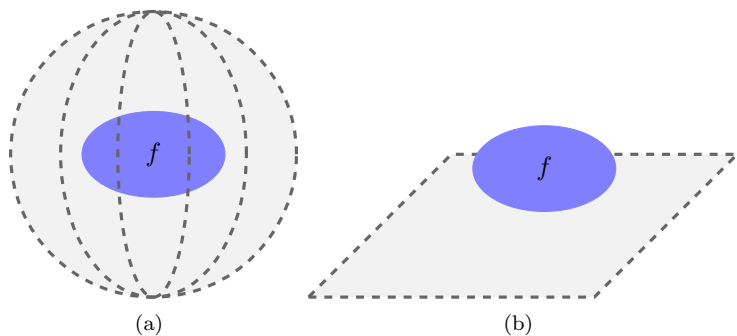
\begin{figure}[h]
    	\centering
    	\subfigure[]{\begin{tikzpicture}[scale=0.95]
    		\draw[dashed][color=black!60, fill=black!5, very thick](-1,0) circle (2);
    		\draw[dashed][color=black!60, very thick](-1,0) ellipse (1.3 and 2);
    		\draw[color=blue!50, fill=blue!50] (-1,0) ellipse (1 and 0.6);
    		\draw[dashed][color=black!60, very thick](-1,0) ellipse (0.5 and 2);
    		
    		\node at (-1,0) {$f$};
    		\end{tikzpicture}}
    	\subfigure[]{\begin{tikzpicture}[scale=0.95]
    		\draw[dashed][color=black!60, fill=black!5, very thick] (3,-2) -- (7,-2) -- (9,0) -- (5,0) -- cycle;
    		\draw[color=blue!50, fill=blue!50] (6.3,-0.2) ellipse (1 and 0.6);
    		
    		\node at (6.3,-0.2) {$f$};
    		\end{tikzpicture}}
    	\caption{PAT detection geometries in $\RR^{3}$. (a) spherical (b) planar.}  	
    	\label{fig:geom}
    \end{figure}

	In  spherical geometry, the solution $p_{\alpha}$ of (1) is measured on the unit sphere $\mathbb{S}^{n-1}$ in $\RR^{n}$. Let us the wave forward operator $\mathcal W_S$ defined by $\mathcal{W}_{S}f(\ttheta,t;\alpha)=p_{\alpha}(\ttheta,t),$ $(\ttheta,t)\in\mathbb{S}^{n-1}\times[0,\infty),$ where $f$ is an initial function of (1).  
	
	Similar to the spherical geometry, the solution $p_{\alpha}$ of (1) is measured on the hyperplane $\{\xx=(\boldsymbol{x}_{*},x_n)\in\RR^{n} : x_n=0, \ \boldsymbol{x}_{*} \in \RR^{n-1} \}.$ 
	Similarly, let us the wave forward operator $\mathcal{W}_H$  defined by  $\mathcal{W}_{H}f(\boldsymbol{u},t;\alpha)=p_{\alpha}(\boldsymbol{u},t)$, $(\boldsymbol{u},t) \in \RR^{n-1} \times [0,\infty),$ where $f$ is an initial function  of (1). 
	
	In both geometries, the Mellin transform is essential to  finding the initial function $f$ from measurement data. Moreover, spherical harmonics are employed in spherical geometry.
	The following subsections are devoted to introductions of the Mellin transform and spherical harmonics.
	
	\subsection{The Mellin transform}
	
	Most of this subsection comes from \cite[p.79$\sim$90]{paris01}.
    Let $f$ be a locally integrable function defined on $(0,\infty).$ The Mellin transform of $f$ is defined by
	\begin{equation}\label{eq:mellin}
	\mathcal Mf(s):=\intL\half f(x)x^{s-1}{\rm d}x, \quad s \in \mathbb{C},
	\end{equation}
	when the integral converges. Suppose that 
	\begin{equation*}
	f(x)=O(x^{-a-\epsilon}) \ \  \mbox{as} \ \  x \rightarrow 0^{+} \quad \mbox{and} \quad f(x)=O(x^{-b+\epsilon}) \ \  \mbox{as} \ \ x \rightarrow \infty
	\end{equation*}
	where $O$ is the Big $O$ notation, $\epsilon>0$ and $a<b$,  integral \eqref{eq:mellin} converges absolutely and defines an analytic function in the strip $a<\mathrm{Re}(s)<b.$
	Furthermore, its inverse transform is given by
	\begin{equation*}\label{eq:inversemellin}
	f(x) = \mathcal M^{-1}(\mathcal Mf)(x) = \dfrac{1}{2\pi\mathrm{i}} \intL^{\gamma+\mathrm{i}\infty}_{\gamma-\mathrm{i}\infty} \mathcal Mf(s)x^{-s}{\rm d}s, \quad \mbox{for} \quad a<\gamma<b.
	\end{equation*}
    Then $f$ can be recovered from its Mellin transform $\mathcal Mf$ using the inverse Mellin transform.
	The Mellin transform satisfies the property:
	\begin{equation*}\label{eq:MfMg}
	\mathcal M(f\times g)(s)=\mathcal Mf(s)\mathcal Mg(s),
	\end{equation*}
	where the convolution is defined by
	\begin{equation}\label{eq:Mellinconvol}
	f \times g(x):=\intL\half f(\tau)g\left(\frac{x}{\tau}\right)\frac{{\rm d}\tau}{\tau}.
	\end{equation}
	
	\subsection{Spherical harmonics}
	
	Let $\mathbi{Y}_{lk}$ denote by the spherical harmonics\cite{natterer86,seeley66} that form a complete orthonormal system in $L^2(\mathbb{S}^{n-1})$. Then, $f$ can be expanded in the spherical harmonics as:
	\begin{equation*}
	f(r_{\xx}\ttheta_{\xx})=\displaystyle\sum^\infty_{l=0}\displaystyle\sum^{\mathbi{N(n,l)}}_{k=0}f_{lk}(r_{\xx})\mathbi{Y}_{lk}(\ttheta_{\xx}),\quad \mbox{for all} \quad f \in L^2(\RR^n), 
	\end{equation*}
	where $N(n,l)=(2l+n-2)(n+l-3)!/(l!(n-2)!)$ for $l\in\mathbb{N}$ and $N(n,0)=1$.
	Also, we use the spherical harmonics expansions of the measurement data $\mathcal W_{S}f(\ttheta,t;\alpha)$ and the Fourier transform $\mathcal{F}f(\xxi)$ of the initial function $f$ with $\xxi=\lambda_{\xxi}\oomega_{\xxi}$,
	\begin{equation}\label{eq:sheofWsf}
	\mathcal{W}_{S}f(\ttheta,t;\alpha) = \displaystyle\sum^\infty_{l=0} \displaystyle\sum^\mathbi{N(n,l)}_{k=0}(\mathcal{W}_{S}f)_{lk}(t;\alpha) \mathbi{Y}_{lk}(\ttheta), \quad \mbox{for all} \quad (t,\ttheta) \in [0,\infty) \times \mathbb{S}^{n-1}
	\end{equation}
	and
	\begin{equation*}\label{eq:sheoffourier}
	\mathcal{F}f(\lambda_{\xxi}\oomega_{\xxi})=\displaystyle\sum^\infty_{l=0}\displaystyle\sum^\mathbi{N(n,l)}_{k=0}(\mathcal{F}f)_{l k}(\lambda_{\xxi})\mathbi{Y}_{l k}(\oomega_{\xxi}),\quad \mbox{for all} \quad (\lambda_{\xxi},\oomega_{\xxi}) \in [0,\infty) \times \mathbb{S}^{n-1}.
	\end{equation*}	
	
	\section{How to recover the initial function}
	
	To recover the initial function, we assume that the point-like detectors are located on the unit sphere and hyperplane. Below we provide a method to recover the initial function $f$ from the solution of the fractional wave equation measured on two geometries. 
	
	\subsection{Spherical geometry}\label{subsec:sphere}
	
	This section provides answer  to recovering the initial function $f$ from $\mathcal{W}_{S}f.$ Here, the measurement data $\mathcal{W}_{S}f$ is given as  
	\begin{equation}\label{eq:solpdeWsf}
	\mathcal{W}_{S}f(\ttheta,t;\alpha)=\frac{1}{(2\pi)^n}\displaystyle\intL_{\RR^n} E_{\alpha}(-t^\alpha|\xxi|^{\alpha})e^{\mathrm{i}\xxi\cdot\ttheta}\mathcal{F} f(\xxi){\rm d}\xxi,\quad\mbox{for}\quad(\ttheta,t)\in\mathbb{S}^{n-1}\times[0,\infty).
	\end{equation}
First,	we consider a relation between $(\mathcal{W}_{S}f)_{lk}$ and $(\mathcal{F}f)_{lk}$ first.
	\begin{lem}\label{lem:Wsfcoeff}
	For $f\in C^\infty(\RR^n)$ with compact support, we have
		\begin{equation}\label{eq:Wsfcoeff}
		(\mathcal{W}_{S}f)_{lk}(t;\alpha)= \dfrac{\mathrm{i}^l}{(2\pi)^{\frac{n}{2}}}\intL\half E_\alpha(-t^{\alpha}\lambda_{\xxi}^{\alpha})(\mathcal{F}f)_{lk} (\lambda_{\xxi})\lambda_{\xxi}^{\frac{n}{2}} J_{l+\frac{n-2}2}(\lambda_{\xxi}) {\rm d}\lambda_{\xxi},
		\end{equation}		
	where $J_{\nu}(\cdot)$ is the  Bessel function of the first kind of order $\nu$.
	\end{lem}
	\begin{proof}
	Changing the variables $\xxi\rightarrow\lambda_{\xxi}\oomega_{\xxi}$ in \eqref{eq:solpdeWsf}, we  write the measurement data as 
		\begin{equation*}\label{eq:newWsf}
		\begin{array}{ll}
		\mathcal{W}_{S}f(\ttheta,t;\alpha)
		&=\dfrac1{(2\pi)^n}\displaystyle\intL_{\mathbb{S}^{n-1}}\displaystyle\intL\half E_\alpha(- t^\alpha \lambda_{\xxi}^{\alpha})e^{\mathrm{i}\lambda_{\xxi}\oomega_{\xxi}\cdot \ttheta}\mathcal F f(\lambda_{\xxi}\oomega_{\xxi})\lambda_{\xxi}^{n-1}{\rm d}\lambda_{\xxi}{\rm d}S(\oomega_{\xxi})\\ 
		&=\dfrac1{(2\pi)^n}\displaystyle\sum^\infty_{l=0}\displaystyle\sum^\mathbi{N(n,l)}_{k=0}\displaystyle\intL_{\mathbb{S}^{n-1}}\displaystyle\intL\half E_\alpha(- t^\alpha \lambda_{\xxi}^{\alpha})e^{\mathrm{i}\lambda_{\xxi}\oomega_{\xxi}\cdot \ttheta}(\mathcal{F}f)_{l k}(\lambda_{\xxi})\mathbi{Y}_{l k}(\oomega_{\xxi})\lambda_{\xxi}^{n-1}{\rm d}\lambda_{\xxi}{\rm d}S(\oomega_{\xxi}) \\
		&=\dfrac{\mathrm{i}^l}{(2\pi)^{\frac{n}{2}}}\displaystyle\sum^\infty_{l=0}\displaystyle\sum^\mathbi{N(n,l)}_{k=0}\intL\half E_\alpha(-t^\alpha\lambda_{\xxi}^{\alpha} )(\mathcal F f)_{l k}(\lambda_{\xxi})\lambda_{\xxi}^{\frac{n}2}J_{l+\frac{n-2}2}(\lambda_{\xxi}){\rm d}\lambda_{\xxi}\mathbi{Y}_{l k}(\ttheta),
		\end{array}
		\end{equation*}
	where 
	in the last line, we used the  Funk-Hecke theorem 
	\cite[(3.19) in chapter 7.]{natterer86}: 
	    \begin{equation}\label{eq:funkhecke}
	    \displaystyle\int_{\mathbb{S}^{n-1}}e^{\mathrm{i}\lambda_{\xxi}\oomega_{\xxi}\cdot \ttheta}\mathbi{Y}_{lk}(\oomega_{\xxi}){\rm d}S(\oomega_{\xxi}) =(2\pi)^{\frac n2}\mathrm{i}^l \lambda_{\xxi}^\frac{2-n}{2} J_{l+\frac{n-2}{2}}(\lambda_{\xxi})\mathbi{Y}_{l k}(\ttheta).
        \end{equation}
        Comparing \eqref{eq:sheofWsf} completes our proof.
	\end{proof}
    
	\begin{thm}\label{thm:mainsphere}
	For $f\in C^\infty(\RR^n)$ with compact support, we have  
		\begin{equation}\label{eq:FofMellin}
		\mathcal M(F_{lk})(s)=2^{\frac{n}{2}}\pi^{\frac{n-2}{2}}\alpha \mathrm{i}^{-l}\Gamma(1-s)\sin\left(\frac {\pi s}\alpha\right) \mathcal M \left[(\mathcal W_{S}f)_{lk}(\cdot;\alpha)\right](s),\quad 0<\mathrm{Re}(s)<\alpha,
		\end{equation}
	where
		\begin{equation*}\label{eq:Flk}
		F_{lk}(\rho)=(\mathcal{F}f)_{l k}(\rho^{-1}) J_{l+\frac{n-2}2}(\rho^{-1})\rho^{-\frac{n+2}{2}}.
		\end{equation*}	
	\end{thm}
	\begin{proof}
    By changing the variables $\lambda_{\xxi} \rightarrow \tilde{\lambda_{\xxi}}^{-1},$ \eqref{eq:Wsfcoeff} can be represented as 
        \begin{equation}\label{eq:Wsf FxE}
        \begin{array}{ll}
        (\mathcal W_Sf)_{lk}(t;\alpha)
        &=\dfrac{\mathrm{i}^l}{(2\pi)^{\frac{n}{2}}}\displaystyle\intL\half E_\alpha(-t^\alpha\lambda_{\xxi}^{\alpha} )(\mathcal{F} f)_{lk}(\lambda_{\xxi})J_{l+\frac{n-2}2}(\lambda_{\xxi})\lambda_{\xxi}^\frac{n}2 {\rm d}\lambda_{\xxi} \\
        &=\dfrac{\mathrm{i}^l}{(2\pi)^{\frac{n}{2}}}\displaystyle\intL\half E_\alpha\left(-t^{\alpha}\tilde{\lambda_{\xxi}}^{-\alpha}\right) \left(\mathcal{F} f\right)_{lk} \left(\tilde{\lambda_{\xxi}}^{-1}\right) J_{l+\frac{n-2}2} \left(\tilde{\lambda_{\xxi}}^{-1}\right) \tilde{\lambda_{\xxi}}^{-\frac{n+4}2} {\rm d}\tilde{\lambda_{\xxi}} \\
        &=\dfrac{\mathrm{i}^l}{(2\pi)^{\frac{n}{2}}}F_{l k}\times E(t;\alpha),
        \end{array}
        \end{equation}
    where 
    \begin{equation}\label{eq:E}
    E(\rho;\alpha)=E_\alpha(-\rho^{\alpha}).
    \end{equation}
    To check that the Mellin transform of $(\mathcal{W}_{S}f)_{lk}(\cdot;\alpha)$ in \eqref{eq:Wsf FxE} is well-defined, it suffices to check that the Mellin transforms of $F_{lk}$ and $E$ are well-defined, respectively.
    Let us consider the Mellin transform of $F_{lk}$:  $\mathcal M \left(F_{lk}\right)(s) =\intL\half(\mathcal{F}f)_{lk}(\rho^{-1})J_{l+\frac{n-2}2}(\rho^{-1})\rho^{s-\frac{n+4}{2}}{\rm d}\rho$. 
Notice that
        \begin{equation*}
        F_{lk}(\rho)=O(\rho^{\infty}) \ \  \mbox{as} \ \  \rho \rightarrow 0^{+} \quad \mbox{and} \quad F_{lk}(\rho)=O(\rho^{-l-n}) \ \  \mbox{as} \ \ \rho \rightarrow \infty,
        \end{equation*}
since $J_{\nu}(\tilde{\rho})=O(\tilde{\rho}^{\nu})$ as $\tilde{\rho} \to 0^{+}$ \cite{luchko13}. Therefore $\mathcal M \left(F_{lk}\right)(s)$ is well-defined for $\mathrm{Re}(s)<l+n-\epsilon$.
    Next, we consider the Mellin transform of $E.$ Taking the Mellin transform of $E,$ we obtain the following formula (see \cite[Lemma 9.1]{haubold11})
        \begin{equation}\label{eq:EofMellin}
        \mathcal{M}(E)(s;\alpha)
        =\intL\half E(\rho;\alpha)\rho^{s-1}{\rm d}\rho
        =\frac{\Gamma\left(\frac{s}{\alpha}\right)\Gamma\left(1-\frac{s}{\alpha}\right)}{\alpha\Gamma(1-s)}=\dfrac{\pi}{\alpha\Gamma(1-s)\sin(\frac{\pi s}{\alpha})}, \qquad 0<\mathrm{Re}(s)<\alpha
        \end{equation} 
    where in the third equality, we applied the Euler's reflection formula $\Gamma(p)\Gamma(1-p)=\dfrac{\pi}{\sin(\pi p)}.$
    Thus the Mellin transform of \eqref{eq:Wsf FxE} is well-defined for $0<\mathrm{Re}(s)<\alpha$. 
    Taking the Mellin transforms on both sides of \eqref{eq:Wsf FxE}, we have
        \begin{equation*}
        \begin{array}{ll}
        \mathcal{M} \left[(\mathcal{W}_{S}f)_{lk}(\cdot;\alpha)\right](s)
        =\dfrac{\mathrm{i}^l}{(2\pi)^{\frac{n}{2}}}\mathcal{M}(F_{lk})(s) \mathcal{M}(E)(s;\alpha)
        =\dfrac{\pi\mathrm{i}^l}{(2\pi)^{\frac{n}{2}}\alpha}\dfrac{\mathcal M(F_{lk})(s)}{\Gamma(1-s)\sin\left(\frac{\pi s}{\alpha}\right)},
        \end{array}
        \end{equation*} 
    where in the second equality, we used \eqref{eq:EofMellin}.
    \end{proof}
    Now taking the inverse Mellin transform of $\mathcal{M} \left[(\mathcal{W}_{S}f)_{lk}(\cdot;\alpha)\right](s)$, we can reconstruct $F_{lk}$ and $(\mathcal Ff)_{lk}$.
	\begin{cor}\label{cor:Flk}
		For $f\in C^\infty(\RR^n)$ with compact support, we reconstruct $f_{lk}$ from $(\mathcal{W}_{S}f)_{lk}$ by the recovery of $F_{lk}:$	
	    \begin{equation*}
	    (\mathcal Ff)_{lk}(\rho)=2^{\frac{n}{2}}\pi^{\frac{n-2}{2}}\alpha \mathrm{i}^{-l}\mathcal{M}^{-1}\left[\Gamma(1-\cdot)\sin\left(\dfrac{\pi \cdot}{\alpha}\right)\mathcal{M} \left[(\mathcal{W}_{S}f)_{lk}\right](\cdot)\right](\rho^{-1})J_{l+\frac{n-2}2}(\rho^{})^{-1}\rho^{-\frac{n+2}{2}}.
	    \end{equation*}		
	\end{cor}
	
	So far, we have considered the measurement data $\mathcal{W}_{S}f$. 
	Note that our approach can be applied to the direction dependent measurement data $g$ (see \cite{zangerl19}.)
		\begin{rmk}
		For $f\in C^\infty(\RR^n)$ with compact support, let 
		\begin{equation*}\label{eq:g}
		g(\ttheta,t;\alpha)=c_1\mathcal{W}_{S}f(\ttheta,t;\alpha)+c_2\left[\ttheta\cdot\nabla_\xx\mathcal{W}_{S}f(\xx,t;\alpha)\right]_{\xx=\ttheta},\quad (\ttheta,t)\in\mathbb{S}^{n-1}\times[0,\infty)
		\end{equation*}
		be the direction dependent measurement data modeled as in \cite[see (1.2)]{zangerl19},
		where $\ttheta\cdot\nabla_\xx\mathcal{W}_{S}f$ is the normal derivative of  $\mathcal{W}_{S}f$ and $c_{1}$, $c_{2}\in\RR$ are constants. Using \eqref{eq:sheofWsf}, \eqref{eq:funkhecke}, and the Bessel function identity $\frac{{\rm d}}{{\rm d}\lambda}\left[ \lambda^{-\nu}J_{\nu}(\lambda)\right]=-\lambda^{-\nu}J_{\nu+1}(\lambda)$ (see, \cite[(5.13) on p.133]{folland09}), we have $g_{lk}$:
		\begin{equation}\label{eq:glk}
		\begin{array}{ll}
		g_{lk}(t;\alpha)
		&=\displaystyle\dfrac{\mathrm{i}^l}{(2\pi)^{\frac{n}{2}}} \displaystyle\intL\half E_\alpha(-t^\alpha\lambda^{\alpha})(\mathcal F f)_{l k}(\lambda)\lambda^\frac{n}2 \left[(c_1+c_2l)J_{l+\frac{n-2}2}(\lambda)-c_2\lambda J_{l+\frac n2}(\lambda)\right]{\rm d}\lambda \\
		&=\dfrac{\mathrm{i}^l}{(2\pi)^{\frac{n}{2}}}F_{l k}\times E(t;\alpha),
		\end{array} 
		\end{equation}
		where 
		we used \eqref{eq:E} and 
		\begin{equation*}
		F_{lk}(\rho)=(\mathcal{F}f)_{l k}(\rho^{-1})\rho^{-\frac{n+2}{2}} \left[(c_1+c_2l)J_{l+\frac{n-2}{2}}(\rho^{-1})-c_2(\rho^{-1}) J_{l+\frac{n}{2}}(\rho^{-1})\right].
		\end{equation*}
		By Taking the Mellin transform on both sides of \eqref{eq:glk}, and since $\mathcal M (g_{lk})$ is well-defined for $0<\mathrm{Re}(s)<\alpha$, we have $\mathcal M(F_{lk}):$	
		\begin{equation*}
		\mathcal M(F_{lk})(s)=2^{\frac{n}{2}}\pi^{\frac{n}{2}-1} \alpha \mathrm{i}^{-l}\Gamma(1-s)\sin\left(\frac {\pi s}\alpha\right) \mathcal M(g_{lk})(s).
		\end{equation*}
		Also, using the inverse Mellin transform of $\mathcal M(F_{lk})$, we recover $F_{lk}$, $\mathcal Ff_{lk}$, and $f$ from the Mellin transform $\mathcal M(F_{lk})$.
	\end{rmk}
	
	\subsection{Hyperplanar geometry}
	
	Similar to the previous subsection \ref{subsec:sphere}, we show how to recover $f$ from $\mathcal W_H f.$  
	The measurement data $\mathcal W_H f$ is given as 
	\begin{equation}\label{eq:defwhf}
	\mathcal W_Hf(\boldsymbol{u},t;\alpha) = \dfrac1{(2\pi)^n} \displaystyle\intL_{\mathbb{R}^{n-1}} \displaystyle\intL_{\mathbb{R}} E_\alpha(-t^{\alpha}\left\vert (\xxi_{*},\xi_{n}) \right\vert^{\alpha}) e^{\mathrm{i}\boldsymbol{u} \cdot \xxi_{*}} \mathcal Ff(\xxi_{*},\xi_{n}){\rm d}\xi_{n}{\rm d}\xxi_{*}, \quad (\boldsymbol{u},t) \in \RR^{n-1} \times [0,\infty).
	\end{equation}
	
First we see the analog of the Fourier slice theorem:	
	\begin{lem}\label{lem:Fu(WHf)}
		For $f\in C^\infty(\RR^n)$ with compact support, we have
		\begin{equation}\label{eq:fuwhf}
		\mathcal F_{\boldsymbol{u}} (\mathcal W_Hf)(\boldsymbol{\eta_{*}},t;\alpha)=\dfrac1{\pi}\displaystyle\intL_{0}^{\infty} E_\alpha(-t^{\alpha}\lambda^{\alpha}) \mathcal Ff(\boldsymbol{\eta_{*}}, \sqrt{\lambda^{2} - \left\vert \boldsymbol{\eta_{*}} \right\vert^{2}}) \dfrac{\lambda \chi_{\left\vert \boldsymbol{\eta_{*}} \right\vert \le  \lambda }(\lambda)}{\sqrt{\lambda^{2}-\left\vert \boldsymbol{\eta_{*}} \right\vert^{2}}}{\rm d}\lambda.
		\end{equation}		
	\end{lem}
    This Lemma for $\alpha=2$ is already studied in \cite{anastasiozmr07,kostlifbw01,moon18}.
	\begin{proof}
		Taking the $n-1$-dimensional Fourier transform of $\mathcal W_Hf$ defined in \eqref{eq:defwhf} with respect to $\boldsymbol{u}$, we have
		\begin{equation*}
		\begin{array}{ll}
		\mathcal F_{\boldsymbol{u}} (\mathcal W_Hf)(\boldsymbol{\eta_{*}},t;\alpha)
		&=\dfrac1{2\pi}\displaystyle\intL_{\RR}E_\alpha(-t^{\alpha}\left\vert (\boldsymbol{\eta_{*}},\xi_{n}) \right\vert^{\alpha})\mathcal Ff(\boldsymbol{\eta_{*}},\xi_{n}){\rm d}\xi_{n} \\
		&=\dfrac1{\pi}\displaystyle\intL_{0}^{\infty}E_\alpha(-t^{\alpha}\left\vert (\boldsymbol{\eta_{*}},\xi_{n}) \right\vert^{\alpha})\mathcal Ff(\boldsymbol{\eta_{*}},\xi_{n}){\rm d}\xi_{n} \\
		&=\dfrac1{\pi}\displaystyle\intL_{0}^{\infty} E_\alpha(-t^{\alpha}\lambda^{\alpha}) \mathcal Ff(\boldsymbol{\eta_{*}}, \sqrt{\lambda^{2} - \left\vert \boldsymbol{\eta_{*}} \right\vert^{2}}) \dfrac{\lambda \chi_{\left\vert \boldsymbol{\eta_{*}} \right\vert \le  \lambda }(\lambda)}{\sqrt{\lambda^{2}-\left\vert \boldsymbol{\eta_{*}} \right\vert^{2}}}{\rm d}\lambda
		\end{array}	
		\end{equation*}
	    where in the second line, we used the evenness of $f$ and $E_{\alpha}$ with respect to the last variable $\xi_{n}$, and in the last line, we changed the variables $\left\vert (\boldsymbol{\eta_{*}},\xi_{n}) \right\vert \rightarrow \lambda.$
	\end{proof}
	
	\begin{thm}\label{thm:mainhyperplane}
		For $f\in C^\infty(\RR^n)$ with compact support, we have
		\begin{equation*}
		\mathcal M (F_{\boldsymbol{\eta_{*}}})(s)=\alpha \Gamma(1-s)\sin\left(\frac {\pi s}\alpha\right)\mathcal M \left[\mathcal F_{\boldsymbol{u}} (\mathcal W_Hf)\right](\boldsymbol{\eta_{*}},s;\alpha),\quad 0<\mathrm{Re}(s)<\alpha
		\end{equation*}
		where
		\begin{equation*}\label{eq:abbreviations2}
		F_{\boldsymbol{\eta_{*}}}(\lambda)=\mathcal{F}f(\boldsymbol{\eta_{*}}, \sqrt{\lambda^{-2}-\left\vert\boldsymbol{\eta_{*}}\right\vert^{2}}) \dfrac{\chi_{\left\vert\boldsymbol{\eta_{*}}\right\vert\le\lambda^{-1}}(\lambda^{-1})}{\lambda^{2}\sqrt{\lambda^{-2}-\left\vert\boldsymbol{\eta_{*}} \right\vert^{2}}}.
		\end{equation*}
		\begin{proof}
			By changing the variables $\lambda\rightarrow\tilde{\lambda}^{-1},$ \eqref{eq:fuwhf} can be represented as
			\begin{equation}\label{eq:WHf FxE} 
			\begin{array}{ll}
			\mathcal{F}_{\boldsymbol{u}}(\mathcal{W}_{H}f)(\boldsymbol{\eta_{*}},t;\alpha)
			&=\dfrac{1}{\pi}\displaystyle\intL_{0}^{\infty}E_{\alpha}\left(-t^{\alpha}\tilde{\lambda}^{-\alpha}\right)\mathcal{F}f(\boldsymbol{\eta_{*}},\sqrt{\tilde{\lambda}^{-2}-\left\vert\boldsymbol{\eta_{*}}\right\vert^{2}}) \dfrac{\chi_{\left\vert\boldsymbol{\eta_{*}}\right\vert\le\tilde{\lambda}^{-1}}(\tilde{\lambda}^{-1})}{\tilde{\lambda}^{3}\sqrt{\tilde{\lambda}^{-2}-\left\vert\boldsymbol{\eta_{*}}\right\vert^{2}}}{\rm d}\tilde{\lambda} \\
			&=\dfrac{1}{\pi}F_{\boldsymbol{\eta_{*}}}\times E(t;\alpha),
			\end{array}
			\end{equation}
			where in the second line, we used the convolution \eqref{eq:Mellinconvol} and \eqref{eq:E}.
			To show that  the Mellin transform of $\mathcal F_{\boldsymbol{u}}(\mathcal W_Hf)$ defined in \eqref{eq:WHf FxE} is well-defined, we  need only check that the Mellin transforms of $F_{\boldsymbol{\eta_{*}}}$ is well-defined, since by  Theorem \ref{thm:mainsphere}  $\mathcal M (E)$ is well-defined for $0<\mathrm{Re}(s)<\alpha$.
			Taking the Mellin transform of $F_{\boldsymbol{\eta_{*}}}$ with respect to $\lambda$, we have
			$\mathcal M\left(F_{\boldsymbol{\eta_{*}}}\right)(s)
			=\intL\half \mathcal{F}f\left(\boldsymbol{\eta_{*}}, \sqrt{\lambda^{-2}-\left\vert\boldsymbol{\eta_{*}}\right\vert^{2}}\right)\frac{\chi_{\left\vert\boldsymbol{\eta_{*}}\right\vert\le\lambda^{-1}}(\lambda^{-1})}{\lambda^{2}\sqrt{\lambda^{-2}-\left\vert\boldsymbol{\eta_{*}}\right\vert^{2}}}\lambda^{s-1}{\rm d}\lambda$.
			We notice that
        \begin{equation*}
        F_{\boldsymbol{\eta_{*}}}(\lambda)=O(\lambda^{\infty}) \ \  \mbox{as} \ \  \lambda \rightarrow 0^{+} \quad \mbox{and} \quad  F_{\boldsymbol{\eta_{*}}}(\lambda)=O(\lambda^{-\infty}) \ \  \mbox{as} \ \ \lambda\rightarrow \infty,
        \end{equation*}
			Therefore $\mathcal M \left(F_{\boldsymbol{\eta_{*}}}\right)$ is well-defined for any $s\in \mathbb C$ and thus the Mellin transform of $\mathcal F_{\boldsymbol{u}}(\mathcal W_Hf)(s)$ is well-defined for $0<\mathrm{Re}(s)<2$.
			Taking the Mellin transform, 
			we have
			\begin{equation*}
			\mathcal{M}\left[\mathcal{F}_{\boldsymbol{u}} (\mathcal{W}_{H}f)\right](\boldsymbol{\eta_{*}},s;\alpha) = \dfrac{1}{\pi} \mathcal{M}(F_{\boldsymbol{\eta_{*}}})(s)\mathcal{M} (E)(s;\alpha)=\dfrac{\mathcal{M}(F_{\boldsymbol{\eta_{*}}})(s)}{\alpha\Gamma(1-s)\sin\left(\dfrac{\pi s}{\alpha}\right)}.
			\end{equation*}
		\end{proof}
	\end{thm}
Again, taking the inverse Mellin transform of $\mathcal{M}(F_{\boldsymbol{\eta_{*}}})(s)$, we  reconstruct $F_{lk}$ and $(\mathcal Ff)_{lk}$.
    \begin{cor}\label{cor:F}
    	For $f\in C^\infty(\RR^n)$ with compact support, we reconstruct $(\mathcal{F}f)_{lk}$ from $(\mathcal{W}_{S}f)_{lk}$ by the recovery of $F_{\boldsymbol{\eta_{*}}}$: for $\boldsymbol\eta=(\boldsymbol{\eta_{*}}, \eta_n)\in\mathbb R^{n-1}\times\mathbb R$,	
    	\begin{equation*}
    	\begin{array}{ll}
    		\mathcal{F}f(\boldsymbol\eta)=F_{\boldsymbol{\eta_{*}}}(|\boldsymbol\eta| ^{-1})\dfrac{\eta_n}{|\boldsymbol\eta|^{2}}=\dfrac{\eta_n}{|\boldsymbol\eta|^{2}}\mathcal{M}^{-1}[\alpha\Gamma(1-\cdot)\sin\left(\dfrac{\pi \cdot}{\alpha}\right)\mathcal{M}\left[\mathcal{F}_{\boldsymbol{u}} (\mathcal{W}_{H}f)\right](\boldsymbol{\eta_{*}},\cdot;\alpha)](|\boldsymbol\eta|^{-1}).
\end{array}
    	\end{equation*}		
    \end{cor}
	
\section{Conclusions and Summary}

In this study we provide how to recover the initial function $f$ from the solutions of the fractional wave equation restricted on the sphere and hyperplane.
We summarize both cases as follows:\\
\textbf{Summary} for the spherical case. 
We can recover $f$ from $\mathcal{W}_{S}f$ in the following steps:
\begin{enumerate}
	\item Find $(\mathcal{W}_{S}f)_{lk}$ using the spherical harmonics (see Lemma \ref{lem:Wsfcoeff}.).
	\item Take the Mellin transform of $(\mathcal{W}_{S}f)_{lk},$ and observe  $\mathcal{M} \left[(\mathcal{W}_{S}f_{\alpha})_{lk}(\cdot)\right](s)$ is well-defined for $0<\mathrm{Re}(s)<\alpha$ (see Theorem \ref{thm:mainsphere}.).
	\item From theorem \ref{thm:mainsphere}, we find $\mathcal{M}(F_{lk})$ from $\mathcal{M}\left[(\mathcal{W}_{S}f)_{lk}\right]$.
	\item Taking the inverse Mellin transform, we recover $F_{lk}$ from the Mellin transform $\mathcal{M}(F_{lk})$ (see Corollary \ref{cor:Flk}.).
	\item Next, we find $(\mathcal{F}f)_{lk}$ from $F_{lk}$ and finally get $f$.
\end{enumerate}
\textbf{Summary} for hyperplane case. We can recover $f$ from $\mathcal{W}_{H}f$ using the following steps:
\begin{enumerate}
	\item Take the $n-1$-dimensional Fourier transform of $\mathcal{W}_{H}f$ to get $\mathcal F_{\boldsymbol{u}} (\mathcal W_Hf)$ (see Lemma \ref{lem:Fu(WHf)}).
	\item Take the Mellin transform of $\mathcal{F}_{\boldsymbol{u}}(\mathcal{W}_{H}f),$ and observe that $\mathcal{M}\left[\mathcal{F}_{\boldsymbol{u}}(\mathcal{W}_{H}f)(\cdot)\right]$ is well-defined for $0<\mathrm{Re}(s)<\alpha$ (see Theorem \ref{thm:mainhyperplane}).  
	\item Using Theorem \ref{thm:mainhyperplane} we find $\mathcal{M}(F_{\boldsymbol{\eta_{*}}})$ from $\mathcal{M}\left[\mathcal{F}_{\boldsymbol{u}}(\mathcal W_Hf)\right]$.
	\item Taking the inverse Mellin transform, we recover $F_{\boldsymbol{\eta_{*}}}$ from the Mellin transform $\mathcal M (F_{\boldsymbol{\eta_{*}}})$ (see Corollary \ref{cor:F}.). 
	\item Next, we find $\mathcal Ff$ from $F_{\boldsymbol{\eta_{*}}}$ and finally get $f$.
\end{enumerate}	
%
%
%
	\section*{Acknowledgements}
This work was supported by the National Research Foundation of Korea grant (MSIP) and by the basic science research program
through the National Research Foundation of Korea (NRF) funded by the Korea government (2018R1D1A3B07041149, NRF-2020R1A4A1018190).
	\bibliographystyle{plain}

\end{document}